\RequirePackage{amsthm}

\documentclass[sn-mathphys-num]{sn-jnl}% Math and Physical Sciences Numbered Reference Style 

\usepackage{graphicx}%
\usepackage{multirow}%
\usepackage{amsmath,amssymb,amsfonts}%
\usepackage{mathrsfs}%
\usepackage[title]{appendix}%
\usepackage{xcolor}%
\usepackage{textcomp}%
\usepackage{manyfoot}%
\usepackage{booktabs}%
\usepackage{algorithm}%
\usepackage{algorithmicx}%
\usepackage{algpseudocode}%
\usepackage{listings}%
\theoremstyle{thmstyleone}%
\newtheorem{theorem}{Theorem}[section]%  meant for continuous numbers
\newtheorem{proposition}[theorem]{Proposition}% 
\newtheorem{lemma}[theorem]{Lemma}
\newtheorem{corollary}[theorem]{Corollary}

\theoremstyle{thmstyletwo}%

\theoremstyle{thmstylethree}%
\newtheorem{definition}[theorem]{Definition}%
\newtheorem{example}[theorem]{Example}%
\newtheorem{remark}[theorem]{Remark}%

\newcommand{\CC}{\mathbb C}

\newcommand{\NN}{\mathbb N}

\newcommand{\RR}{\mathbb R}

\newcommand{\cC}{\mathcal{C}}

\newcommand{\cF}{\mathcal{F}}

\newcommand{\cK}{\mathcal{K}}

\newcommand{\cO}{\mathcal{O}}

\newcommand{\im}{\operatorname{Im\,}}
\newcommand{\re}{\operatorname{Re\,}}

\newcommand{\Ker}{\textup{\textbf{Ker}}}
\newcommand{\preker}{\operatorname{pre-ker}}

\begin{document}

\title[A Selection Theorem for pointed Domains]{A Selection Theorem for 
the Carath\'eodory Kernel Convergence of Pointed Domains}

%%=============================================================%%
%% GivenName	-> \fnm{Joergen W.}
%% Particle	-> \spfx{van der} -> surname prefix
%% FamilyName	-> \sur{Ploeg}
%% Suffix	-> \sfx{IV}
%% \author*[1,2]{\fnm{Joergen W.} \spfx{van der} \sur{Ploeg} 
%%  \sfx{IV}}\email{iauthor@gmail.com}
%%=============================================================%%

\author[1]{\fnm{Kang-Tae} \sur{Kim}}\email{kimkt@postech.ac.kr}

\author[2]{\fnm{Thomas} \sur{Pawlaschyk}}\email{pawlaschyk@uni-wuppertal.de}

%\equalcont{These authors contributed equally to this work.}

%\author[1,2]{\fnm{Third} \sur{Author}}\email{iiiauthor@gmail.com}
%\equalcont{These authors contributed equally to this work.}

\affil[1]{\orgdiv{Department of Mathematics}, \orgname{Pohang University of Science and Technology}, \orgaddress{\city{Pohang City}, \postcode{37673}, \country{South Korea}}}

%\state{}, \street{},

\affil[2]{\orgdiv{School of Mathematics and Natural Sciences}, \orgname{University of Wuppertal}, \orgaddress{\street{Gau\ss{}stra\ss{}e 20}, \city{Wuppertal}, \postcode{42119}, \country{Germany}}}

%\affil[3]{\orgdiv{Department}, \orgname{Organization}, \orgaddress{\street{Street}, \city{City}, \postcode{610101}, \state{State}, \country{Country}}}

%%==================================%%
%% Sample for unstructured abstract %%
%%==================================%%

\abstract{We present a selection theorem for domains in $\CC^n$, $n\ge 1$, which 
states that any tamed sequence of pointed connected open subsets 
admits a subsequence convergent to its own kernel in the sense of 
Carath\'eodory. Not only is this analogous to the well-known Blaschke 
selection theorem for compact convex sets, but it fits better in the study of 
normal families of biholomorphic maps with varying domains and ranges.}


%%\pacs[JEL Classification]{D8, H51}


\maketitle

\section*{Acknowledgement} We appreciate the hospitality of 
Pusan National University and Gyeongsang National University for their 
hospitality during the research visits in March 2024. An integral part of 
this work was formed then and there. This work was supported in part 
by Grant No.\ SSTF-BA2201-01 from Samsung Sci. and Tech. 
Foundation (South Korea) and by NRF Grant 
No.\ 2023R1A2C1007227 (South Korea), whose PI are 
Young-Jun Choi (PNU) and Kang-Hyurk Lee (GNU).  We also thank 
Tobias Harz and Jae-Cheon Joo for many valuable discussions. We thank the anonymous referee 
for comments improving the presentation of our paper. This version of the article has been accepted for publication in The Journal of Geometric Analysis, after peer review but is not the final published version and does not reflect post-acceptance improvements, or any corrections. The final version is available online at: \href{https://doi.org/10.1007/s12220-025-02077-2}{ https://doi.org/10.1007/s12220-025-02077-2}.

\section{Introduction}

The celebrated Blaschke selection theorem \cite{Blaschke} states that the
space of nonempty compact convex subsets of a Banach space is 
Cauchy-complete in the Hausdorff distance~\cite{Hausdorff}, which implies 
that any bounded sequence of nonempty compact convex subsets of a 
Banach space contains a convergent subsequence. This theorem has been 
generalized to the broader collection of nonempty compact subsets of 
a Banach space~\cite{Price}.
\smallskip

On the other hand, for the conformal maps from the open unit disc into
$\CC$, each of which assigns the origin to a fixed point, the convergence 
of the image domains in the complex plane requires another concept of 
convergence of sets, suggested by Carath\'eodory~\cite{Cara1912}, 
nowadays known as the \emph{Carath\'eodory kernel convergence}.  
This has turned out to be the correct and optimal concept for the 
study of sequences of general connected open sets (i.e., domains).   
\smallskip

The structure of this paper centers around Theorem~\ref{selection_thm}, 
a selection theorem for the Carath\'eodory kernel convergence of pointed 
domains, as well as Theorem~\ref{GCTh}, which may be regarded as 
a generalized version of the Carath\'eodory 
kernel theorem to all dimensions.

\section{The Carath\'eodory kernel convergence}
\label{section-kernel-convergence}

By a \emph{pointed set} we mean a pair $(G,p)$ consisting of a set $G$ and 
a point $p \in G$. Clearly, the set operators, such as inclusion, union and intersection, naturally transfer to point sets with common point, e.g., $(A,p) \subseteq (B,q)$ if $p=q$ and $A \subseteq B$, etc.

We call it a \emph{pointed domain} if the set $G$ is
a connected open set. We denote by $A^\circ$ the \emph{interior} of the 
set $A$ and by $\operatorname{Conn}_q( A)$ the 
\emph{connected component} of~$A$ containing~$q$. 

\begin{definition} A sequence $\{(G_j, p_j)\}_{j\geq1}$ of pointed 
domains in $\CC^n$ is said to be \emph{tamed} at $\hat p$, if the following 
conditions hold:
\begin{enumerate}
\item $\lim\limits_{j\to\infty} p_j = \hat p$ for some $\hat p \in \CC^n$.
\item There is an open neighborhood of $\hat p$ contained in 
$\bigcap_{j \ge k} G_j$ for some~$k \ge 1$. 
\end{enumerate}
\end{definition}

\begin{definition} \label{kernel}
Given a sequence $\{(G_j, p_j)\}_{j\geq1}$ of pointed domains tamed at 
$\hat p \in \CC^n$, its \emph{Carath\'eodory kernel} (or, its 
\emph{kernel}, for short) is the pointed set defined by
\[
\Ker_{\hat p} \{(G_j, p_j)\}_{j\geq1} 
:= 
\Big( \bigcup\limits_{k\ge 1}  \operatorname{Conn}_{\hat p} 
\Big(\big( \bigcap_{j \ge k} G_j \big)^\circ\Big), ~\hat{p}\Big).
\]
\end{definition}

Next, we define the convergence of pointed domains in the sense of Carath\'eodory~\cite{Cara1912}.

\begin{definition} \label{kernel-convergence}
A sequence $\{(G_j, p_j)\}_{j\geq1}$ of pointed domains tamed at $\hat p \in \CC^n$ is said to
\emph{converge (to its kernel)} if for each subsequence $\{(G_{j_k}, p_{j_k})\}_{k\geq1}$ the kernel is the same, i.e., 
\[
\Ker_{\hat p} \{(G_j, p_j)\}_{j\geq1} 
= \Ker_{\hat p} \{(G_{j_k}, p_{j_k})\}_{k\geq1}.
\]
\end{definition}

\begin{example} \label{Pac-man_example}
Let $\Delta = \{z \in \CC \colon |z|<1\}$ be the open unit disc in $\CC$.  
Then for each $j \geq 1$ put 
\[
G_j := \Delta \setminus \big\{ x+iy \in \CC\colon  -1<x<1-\tfrac1j, 
\ y=0 \big\}
\]   
and let $\hat p := p_j:=\frac{\sqrt{-1}}2 =\frac i2$.  
By the Riemann mapping theorem there exists, for 
every $j>1$, a biholomorphism $f_j$ mapping $\Delta$ onto $G_j$ with 
$f_j (0)=\hat p$ and $f_j '(0)>0$. The sequence of maps $\{f_j\}_{j>1}$ 
contains a subsequence that converges \emph{uniformly on compact subsets} 
(or alternatively, \emph{locally uniformly}) by Montel's theorem.  So
we may assume, by taking a subsequence, that this sequence converges
in the same manner to $\hat f\colon \Delta \to \CC$.  Let us now
find $\hat f (\Delta)$. In the sense of the Hausdorff convergence, it may be 
natural to expect that $\{G_j\}_{j>1}=\{f_j (\Delta)\}_{j>1}$ converges 
to $\Delta$ by, taking first the closure of the $G_j$'s, taking the limit, 
which is the closed unit disc, and finally taking the interior.  
However, it is well known that $\hat f (\Delta)$ is the upper half-disc 
$\Delta^+:=\{x+iy \in \Delta : y > 0\}$. 
In contrast, the convergence of $\mathcal{G}:=\{(G_j, \frac{i}2)\}_{j>1}
= \{(f_j (\Delta), p_j)\}_{j>1}$ in the sense of Carath\'eodory finds 
exactly the correct limit, namely 
$\Ker_{i/2}(\mathcal{G})=(\Delta^+,\frac{i}{2})=(\hat f(\Delta),\hat p)$.
\end{example}

Now let us compare the kernel with another notion of set limit known as the \emph{normal limit}.

\begin{definition} \label{NormalConv}
Given a sequence of pointed domains $\{(G_j, p_j)\}_{j\geq1}$ tamed 
at $\hat p$, a pointed domain 
$(\hat G, \hat p)$ is called its \emph{normal limit}, 
if the following two conditions hold:
\begin{enumerate}
\item 
For every  connected compact subset $K$ of $\hat G$ with 
$\hat p \in K$, there exists $k_0 \ge 1$ such that 
$K$ is contained in $\Big(\bigcap_{j \ge k_0} G_j\Big)^\circ$.
\item If for a compact connected subset $L$ with $\hat{p} \in L$ there 
exists an index $k_1\geq1$ such that $L \subset  
\Big(\bigcap_{j \ge k_1} G_j\Big)^\circ$, then $L$ lies in~$\hat G$.  
\end{enumerate}
\end{definition}

\begin{remark} The concept of such a limit domain was posed earlier in a similar form, 
e.g.~\cite{BedPin} or \cite{GKK2011} (p.~228, Definition~9.2.2.), with 
an indication that the sequence of sets would converge to this limit.  
However, the necessity of taking \emph{connected} compact subsets was 
overlooked in those articles. In fact, Definition~\ref{NormalConv} appears to be closer to the ``\emph{Kern}'' (kernel) introduced by Carath\'eodory in \cite{Cara1912}, but in the same paper, the convergence as in Definition~\ref{kernel-convergence} was required to get results on the convergence of families of conformal maps.
\end{remark}

The relation of the normal limit and the kernel is clarified in the following

\begin{proposition}
A pointed domain $(\hat G, \hat p)$ is the normal limit of the 
sequence $\{(G_j, p_j)\}_{j\geq1}$ of pointed domains tamed at $\hat p$ 
if and only if
$(\hat G, \hat p) = \Ker_{\hat p} \{ (G_j, p_j)\}_{j\geq1}$.
\end{proposition}

\begin{proof}
Assume first that $(\hat G, \hat p)$ is the normal limit of the sequence 
$\{(G_j, p_j)\}_{j\geq1}$. Then condition~(1) in 
Definition~\ref{NormalConv} is equivalent to
\[
\hat p \in K \Subset \operatorname{Conn}_{\hat p} \Big( \Big(\bigcap_{j\ge k}
G_j \Big)^\circ \Big) \text{ for some } k\ge 1.
\]
As a result, $(\hat G,\hat p) \subseteq \Ker_{\hat p} \{(G_j, p_j)\}_{j\geq1}$.
Condition~(2) implies the opposite inclusion.  So 
$(\hat G, \hat p) = \Ker_{\hat p} \{(G_j, p_j)\}_{j\geq1}$. The converse follows from the definition of the Carath\'eodory kernel and by 
the compactness of the involved sets in the definition of the normal limit.
This completes the proof.
\end{proof}

Notice that the existence of the normal limit does not, in general, guarantee the convergence to its kernel.

\begin{example}\label{drunken-ellipsi}
Let
\begin{align*}
G := \{x+iy \in \CC \colon |x|<1,\ |y|<3\}, \\
H := \{x+iy \in \CC \colon |x|<3,\ |y|<1\}.
\end{align*}
Then construct the sequence of domains pointed at the origin such as
\[
G_j := \begin{cases} 
(G, 0), &\text{if } j = \text{odd}, \\
(H, 0), &\text{if } j = \text{even}.
\end{cases}
\]
This is a tamed sequence and its kernel is
\[
\Ker_0 (G_j, 0) = \big(\{ x+iy \in \CC \colon |x|<1, |y|<1\}, 0 \big).
\]
Thus, this is the normal limit.  The sequence, however, does not converge 
to its kernel in the sense of Carath\'eodory, since it contains two constant 
subsequences $\{(G,0)\}$ and $\{(H, 0)\}$, whose kernels are $(G,0)$ and 
$(H,0)$, respectively.  

Consequently, it is evident that the normal limit finds the kernel. but it 
does not necessarily imply the convergence in the sense of Carath\'eodory.
\end{example}

\begin{remark} If we define the normal convergence of a tamed sequence  
by requiring that all subsequences share the same normal limit, then the 
normal convergence is equivalent to the convergence in the sense of 
Carath\'eodory.
\end{remark}

\begin{remark}

The original Carath\'eodory kernel~\cite{Cara1912} for a sequence of 
pointed domains in $\CC$ has been defined even when the sequence is 
not tamed, i.e., when $\hat p$ is not an interior point of 
$\bigcap_{j \geq k} G_j$ for any 
$k \geq 1$.  In such a case, the kernel is defined to be simply 
the singleton set~$\{\hat p\}$. This completes conceptually the definition 
of the kernel, since the degenerate case corresponds to a sequence of 
images $G_j=f_j(D)$ of a \textit{compactly divergent} sequence of holomorphic 
maps $f_j$ defined on a plane domain $D$. This
corresponds to the case that the kernel according to 
Definition 2.2. is the empty set.
\end{remark}

\section{A selection theorem for the kernel convergence}
\label{section-selection-thm}

The present goal is to give a version of the selection theorem for 
tamed sequences of domains with respect to the convergence in the 
sense of Carath\'eodory.

\begin{lemma}[Monotonicity of kernels]\label{monotonicity-kernels}
Let $\sigma = \{(G_j, p_j)\}_{j\geq1}$ be a sequence of pointed domains tamed at $\hat p$ in $\CC^n$. If $\tau$ is a subsequence of~$\sigma$, then their kernels
satisfy
\[
\Ker_{\hat p} (\sigma) \subseteq \Ker_{\hat p} (\tau).
\]
\end{lemma}

\begin{proof}
The proof follows directly from the definition of the kernel.
\end{proof}

%%%%%
\begin{theorem}[Selection theorem for domains] \label{selection_thm}
Let $\sigma=\{(G_j, p_j)\}_{j\geq1}$ be a sequence of pointed domains 
 tamed at $\hat p$ in $\CC^n$. Then there exists a subsequence 
$\tau$ convergent to its kernel $\Ker_{\hat p} (\tau) $.
\end{theorem}

\begin{proof} 
The hypothesis on the tameness of the set sequence implies that the Carath\'eodory kernel
\[
\Ker_{\hat p} (\sigma) 
= 
\Big( \bigcup\limits_{k\ge 1}  
\operatorname{Conn}_{\hat p} \Big(\big( \bigcap_{j \ge k} 
G_j \big)^\circ\Big), ~\hat p\Big)
\]
is nonempty.
Let $\Sigma_\sigma$ be the set of all subsequences of $\sigma$. 
Denote by 
\[
\cK_\sigma = \{ \Ker_{\hat p}(\gamma) \colon \gamma 
\in \Sigma_\sigma\}.
\]
Equipped with the inclusion relation, the set $\cK_\sigma$ becomes a partially 
ordered set. Notice in passing that $\Ker_{\hat p}(\sigma)$ is the 
minimal element.
\smallskip

Recall the set-theoretical concept of a \emph{chain}, i.e., a 
totally-ordered subset.  
Then the \emph{Hausdorff maximum principle} (or, \emph{Zorn's 
lemma}) states that, in every partially ordered set, every nonempty chain 
admits a \emph{maximal} chain. Now, choose a maximal chain 
$\cC_\sigma$ in $\cK_\sigma$ and let
\[
(\hat K, \hat p) := \bigcup_{(K,\hat p) \in \cC_\sigma} (K, \hat p).
\]
Note that $\hat K$ is open and non-empty due to the tameness of 
$\sigma$ at $\hat p$.
\medskip

\noindent \textbf{Claim.} \emph{There exists a subsequence 
$\tau \in \Sigma_\sigma$ such that $(\hat K,\hat p) = \Ker_{\hat p} (\tau)$.}
\medskip

\noindent To justify this claim, take a sequence $\{Q_m\}_{m\geq1}$ 
of connected compact subsets of $\hat K$ satisfying
\[
\hat{p} \in Q_m \Subset Q_{m+1}^\circ \ \text{for every } m \geq 1, 
\quad \text{and} \quad (\hat K,\hat p)=\Big( \bigcup_{m=1}^\infty Q_m, \hat p \Big).
\]
Let $\NN$ be the set of natural numbers. 
For every $\alpha \in \Sigma_\sigma$, denote by 
$\tilde\alpha \colon \NN \to \NN$ the \emph{index function} 
associated with $\alpha$ such that 
\[
\alpha  = \{(G_{\tilde\alpha (k)}, p_{\tilde\alpha (k)})\}_{k\geq1} \quad 
\text{and} \quad k \le \tilde\alpha (k) < \tilde\alpha (k+1) 
\textrm{ for every } k\in \NN.
\]
Now we construct the sequence 
$\tau=\{(G_{j_\ell}, p_{j_\ell})\}_{\ell \geq1}$ 
by induction. 
\smallskip

Let $\ell=1$. Then, by the construction of $\cC_\sigma$, $\hat K$ and 
$\{Q_m\}_{m\geq1}$, there exists a subsequence 
$\alpha_1 \in \Sigma_\sigma$ with 
$\Ker_{\hat p} (\alpha_1) \in \cC_\sigma$ such that 
\[
(Q_1,\hat p) \subseteq  \Ker_{\hat p} (\alpha_1).
\]
By definition of the Carath\'eodory kernel, there is a natural number 
$m_1 \in \NN$ such that 
\[
Q_1 \subseteq \operatorname{Conn}_{\hat p}
\Big( \big(\bigcap_{j \ge m_1} G_{\widetilde{\alpha_1} (j)}\big)^\circ \Big).
\]
So take $j_1 := \widetilde{\alpha_1}(m_1)$.
\medskip

Assume now that we already found the index 
$j_\ell=\widetilde{\alpha_{\ell}}(m_\ell)$ for an $\ell\geq1$, where $
\alpha_\ell$ is a subsequence of $\sigma$ fulfilling 
$\Ker_{\hat p} (\alpha_\ell) \in \cC_\sigma$ and
\[
Q_\ell \subseteq\operatorname{Conn}_{\hat p} 
\Big( \big(\bigcap_{j \ge m_\ell} 
G_{\widetilde{\alpha_\ell} (j)}\big)^\circ \Big)
\]
for some $m_\ell \geq 1$.

Now for ${\ell+1} \in \NN$, there is 
$\alpha_{\ell+1} \in \Sigma_\sigma$ such that 
$\Ker_{\hat p}(\alpha_{\ell+1}) \in \cC_\sigma$ and 
$(Q_{\ell+1},\hat p) \subset \Ker_{\hat p} (\alpha_{\ell+1})$.
Consequently, there is $n_{\ell+1} \in \NN$ such that 
\[
Q_{\ell+1} \subseteq  \operatorname{Conn}_{\hat p} 
\Big( \big(\bigcap_{j \ge n_{\ell+1}} G_{\widetilde{\alpha_{\ell+1}}(j)}
\big)^\circ \Big)
\]
Since 
\[
\operatorname{Conn}_{\hat p} 
\Big( \big(\bigcap_{j \ge \nu} G_{\tilde\beta (j)}\big)^\circ \Big)
\subseteq
\operatorname{Conn}_{\hat p} 
\Big( \big(\bigcap_{j \ge \mu} G_{\tilde\beta (j)}\big)^\circ \Big),
\]
for any subsequence $\beta  \in \Sigma_\sigma$, and 
$\nu,\mu  \in \NN$ with $\nu<\mu$, we may choose 
$m_{\ell+1} \in \NN$ so that $m_{\ell+1} > n_{\ell+1}$ and 
\[
j_{\ell}=\widetilde{\alpha_\ell} (n_\ell) < \widetilde{\alpha_{\ell+1}}  
(m_{\ell+1}).
\]
So we let $j_{\ell+1} := \widetilde{\alpha_{\ell+1}} (m_{\ell+1})$.  
\smallskip

By induction, we obtain a subsequence
\[
\tau = \{(G_{j_\ell}, p_{j_\ell})\}_{\ell \geq1} \in \Sigma_\sigma
\]
which admits
\[
( \hat K,\hat p) = \Big( \bigcup_{j=1}^\infty Q_j,\hat p \Big) \subseteq \Ker_{\hat p} (\tau).
\]
Hence, $\cC_\sigma  \cup \{\Ker_{\hat p} (\tau)\}$ is a chain
in $\cK_\sigma$.  It also contains $\cC_\sigma$.
The maximality of $\cC_\sigma$ implies that 
$\cC_\sigma  \cup \{\Ker_{\hat p} (\tau)\} = \cC_\sigma$ and 
the definition of $\hat K$ implies that $(\hat K,\hat p) = \Ker_{\hat p} (\tau)$.
This proves the claim.
\medskip

To complete the proof of Theorem~\ref{selection_thm}, we still have
to show that the sequence $\tau$, just constructed, converges to its 
Carath\'eodory kernel $(\hat K,\hat p) = \Ker_{\hat p}(\tau)$. 
Let $\eta$ be an arbitrary subsequence of~$\tau$.  By monotonicity 
(Lemma~\ref{monotonicity-kernels}), the kernel of $\eta$ contains 
the kernel of $\tau$.  Then the maximality of $\Ker_{\hat p}(\tau)$
in $\cK_\sigma$ implies that $\Ker_{\hat p}( \tau ) = \Ker_{\hat p}(\eta)$. 
This shows that every subsequence of~$\tau$ shares the same kernel with
$\tau$. Thus, $\tau$ converges to its kernel. 
This now completes the proof of the selection theorem.
\end{proof}

\section{On Carath\'eodory's kernel theorem in all dimensions}
\label{section-cara-kernel-thm}

For two domains $D$ and $G$ in $\CC^n$, denote by $\cO(D, G)$ the family
of holomorphic maps from $D$ into $G$.  A sequence 
$\{f_j\}_{j \geq 1} \subset \cO(D, G)$ is called \emph{compactly divergent} 
on $D$ if, for any $K \Subset D$ and $L \Subset G$, there exists 
$j_0\geq1$ such that $f_j (K) \cap L = \varnothing$ for every 
$j \geq j_0$.  Any subfamily $\cF$ of $\cO(D, G)$ is called a 
\emph{normal family} if every sequence
contains a subsequence that converges locally uniformly, or
a subsequence compactly divergent on $D$. Normal families are closely 
related to \emph{tautness} of domains (cf.\ \cite{Wu}). From here on, 
without exception, the notation $\Delta$ represents
the unit open disc in the complex plane $\CC$.

%\begin{definition}
%A domain $G$ in $\CC^n$ is said to be \emph{taut}, if the family 
%$\cO (\Delta, G)$ of holomorphic maps from the unit open 
%disc~$\Delta$ in $\CC$ into~$G$ is a normal family.
%\end{definition}

\begin{definition}
A domain $G$ in $\CC^n$ is said to be \emph{taut}, if for any complex 
manifold $M$ the collection $\cO (M, G)$ is a 
normal family.
\end{definition}

Then the following result is a combination of Lemma~1.3 in~\cite{Wu} and 
Theorem~2 in-\cite{Barth}.

\begin{theorem}
\label{Barth}
A domain $G \subset \CC^n$ is taut if and only if $\cO (\Delta, G)$ is a 
normal family, where $\Delta$ denotes the unit disc in $\CC$.
\end{theorem}

%\begin{theorem}[\cite{Wu, Barth}] 
%\label{Barth}
%If a domain $G \subset \CC^n$ is taut, then for any domain 
%$D \subset \CC^m$, $m \geq1$, the collection $\cO (D, G)$ is a 
%normal family.
%\end{theorem}

%\begin{proof}If $G$ is taut, it follows immediately that $\cO(\Delta^m,G)$ 
%is a normal family for any $m\geq1$.  Then the proof of this theorem
%follows by the fact that each compact subset of $G$ can be 
%covered by the union of finitely many polydiscs in~$G$.  
%\end{proof}

We recall the notion of Kobayashi hyperbolicity.

\begin{definition}
A domain $G$ in $\CC^n$ is said to be \emph{Kobayashi hyperbolic} if the Kobayashi pseudo-metric $d_G$ on $G$ (cf. \cite{Koba1970, Koba2006}) is a metric, i.e., $d_G(p,q)>0$ for every $p,q \in G$ with $p \neq q$. $G$ is called \emph{complete} if $d_G$ is complete.
\end{definition}

\begin{remark} All bounded domains in $\CC^n$ clearly are taut. All complete Kobayashi hyperbolic manifolds are taut. And all taut mainfolds are hyperbolic (cf. Theorem (5.1.3)~in \cite{Koba2006}). All hyperbolically embedded domains are taut \cite{Kiernan}. \end{remark}

We need the following generalization of Cartan's uniqueness theorem.

\begin{proposition}
\label{GCUT}
Let $G$ be a Kobayashi hyperbolic domain in $\CC^n$ with $p \in G$.  
If a holomorphic map $f\colon G \to G$ satisfies the following two 
conditions
\begin{enumerate}
\item $f(p)=p$,
\item $df_p$ coincides with the identity map,
\end{enumerate}
then $f$ itself coincides with the identity map.
\end{proposition}

\begin{proof}
The proof here is extracted from the most general 
version in~\cite{LeeKH2005}. Take a bounded open neighborhood $U$ 
of $p$ in $G$. Recall that the taut domains are Kobayashi hyperbolic 
(cf.~\cite{Koba1970}), and that the standard topology of~$G$ is 
equivalent to the metric topology of the Kobayashi distance~\cite{Barth1972}.  
Consequently, $U$ is the union of Kobayashi distance open balls contained 
in $U$. In particular, there is $r>0$ such that the open Kobayashi ball, say
$W$, of radius $r$ centered at $p$ is contained in $U$.  Then the 
distance-decreasing property of $f$ yields that $f(W) \subseteq W$.  
Since $W$ is a bounded open region in $\CC^n$, the classical 
Cartan uniqueness theorem implies that $f$ coincides with the 
identity map on $W$.  Then it follows from the identity theorem for 
holomorphic functions that $f$ is the identity map on the whole of $G$, 
which yields the proof.
\end{proof}
\medskip

Now, we present the following high-dimensional analog of the original 
Carath\'eodory kernel theorem  (cf.~\cite{Cara1912}).  
To avoid excessive notation, we denote by 
$f\colon (D, p) \to (G, q)$ the map $f\colon D \to G$
satisfying $f(p)=q$.

\begin{theorem} \label{GCTh} 
Let $\mathcal{D}=\{(D_j, p_j)\}_{j\geq1}$ and 
$\mathcal{G}=\{(G_j, q_j)\}_{j\geq1}$ be sequences of pointed 
domains in $\CC^n$ tamed at $\hat p$ and $\hat q$, respectively, admitting 
taut domains $\widetilde D$ and 
$\widetilde G$ such that $D_j \subset \widetilde D$ and 
$G_j \subset \widetilde G$ for every $j\geq 1$. 
If $\{f_j \colon (D_j, p_j) \to (G_j, q_j)\}_{j\geq1}$ is a sequence of 
biholomorphic maps, then the following hold:
\begin{enumerate}
\item There is a subsequence of 
$\{f_j \colon (D_j, p_j) \to (G_j, q_j)\}_{j\geq1}$ for which the 
corresponding subsequences of $\mathcal{D}$ and $\mathcal{G}$,
respectively, converge to their own kernels $(\hat D, \hat p)$ and
$(\hat G, \hat q)$, respectively.
\item The resulting subsequence of (1) satisfies that every 
subsequential limit (with respect to the compact-open topology) is a 
biholomorphism from $(\hat D, \hat p)$ onto $(\hat G, \hat q)$.
\end{enumerate}
\end{theorem}

Notice that, by monotonicity of the kernels, $\hat f$ is at least defined 
on the kernel of the initial sequence of domains 
$\mathcal{D}=\{(D_j, p_j)\}_{j\geq1}$.

\begin{proof}
Throughout this proof, we are going to take the 
subsequences of $\{f_j\}_{j\geq1}$ successively, as many times 
as necessary. While doing so,  we shall continue using the same 
notation $\{f_j\}_{j\geq1}$ for these subsequences.
\smallskip

By the selection theorem (Theorem \ref{selection_thm}), we take a 
subsequence of $\{f_j\}_{j\geq1}$ so that a subsequence of $\mathcal{D}$ 
converges to its own kernel~$(\hat D, \hat p)$. 
Then we extract a subsequence again so that a subsequence of 
$\mathcal{G}$ also converges to its own kernel~$(\hat G, \hat q)$.
\smallskip

The tautness of $\widetilde D$ and $\widetilde G$ implies that
a subsequence can be extracted for the third time, so that 
$\{f_j\}_{j\geq1}$ converges uniformly on compact subsets to 
a holomorphic map $\hat f$ from $(\hat D, \hat p)$ into
$(\hat G, \hat q)$. Notice that the possibility of 
compactly divergent subsequence is immaterial, since our sequences of
pointed domains under consideration are \textit{tamed}.
Then a subsequence can be extracted for the fourth 
time so that $\{f_j^{-1}\}_{j\geq1}$ converges locally uniformly to a 
holomorphic map $\hat g$ from $(\hat G, \hat q)$ into 
$(\hat D, \hat p)$.  
\smallskip

Let $F := \hat g \circ \hat f$.  Then $F(\hat p)=\hat p$, and 
$dF|_{\hat p} = \lim_{j \to \infty} (df_j|_{\hat p})^{-1} \circ df_j|_{\hat p}$ 
equals the identity map.  Since the domains $\hat G$ and $\hat D$ are 
contained in the taut domains $\widetilde G$ and $\widetilde D$, 
respectively, they are Kobayashi hyperbolic.  
By Proposition~\ref{GCUT}, the map
$F$ coincides with the identity map on~$\hat D$.  Since the same argument
works for $G:= \hat f \circ \hat g$ on~$\hat G$, it follows that 
$\hat f:(\hat D,\hat p) \to (\hat G,\hat q)$ and 
$\hat g:(\hat G,\hat q) \to (\hat D,\hat p)$ are 
biholomorphisms with $\hat{f}^{-1}=\hat g$.  
This completes the proof.
\end{proof}
\medskip

Related to the extension of the limit maps in the previous theorem, we 
present the following

\begin{proposition}\label{multimap_ext}
Let $\{f_j \colon (D_j, p_j) \to (G_j, q_j)\}_{j\geq1}$ be the sequence 
biholomorphic maps and let 
$\hat f \colon (\hat D, \hat p) \to (\hat G, \hat q)$ be the subsequential 
limit given in Theorem~\ref{GCTh}.
Then $\hat f$ extends holomorphically to the union of the maximal
kernels as a multimap.
\end{proposition}

\begin{proof}
Notice that $\Ker_{\hat p} \{(D_j, p_j)\}_{j\geq1}$ is the minimal element 
in the set of the kernels of all subsequences of~$\{(D_j, p_j)\}_{j\geq1}$. 
Any two subsequential limits $\widehat{h}$ and $\widehat{g}$ resulting from 
Theorem~\ref{GCTh} with maximal kernels $(\widehat{D^h}, \hat p)$ and 
$(\widehat{D^g}, \hat p)$, respectively, have the following properties. Firstly, both maximal kernels contain $\Ker_{\hat p} \{(D_j, p_j)\}_{j\geq1}$ 
as a subset. Secondly, they admit a complex affine biholomorphism $A$ of $\CC^n$ 
such that $\hat h = A \circ \hat g$ at every point of the kernel
$\Ker_{\hat p} \{(D_j, p_j)\}_{j\geq1}$. This proves the assertion.
\end{proof}
\medskip

\begin{remark} \label{ex-flying-saucer}
There have been some suggestions that it might suffice to assume
that the kernels of the pointed domains are taut, or even 
complete Kobayashi hyperbolic.  Consider the sequence of the following 
domains pointed at the origin and stretching to infinity along the 
$z$-axis, defined by 
\[
S_j = \{(z,w) \in \CC^2:|z|^2 + |w|^2 + \frac1j \log |w|^2 < 1\}
\]
for $j\geq1$. This sequence converges to its kernel,
the open unit ball in $\CC^2$, which is completely hyperbolic.  
On the other hand, the normal family arguments fail for the maps 
from $S_j$ to itself.
\end{remark}

\section{On computations of kernels}
\label{section-kernel-computations}

We present some methods of computing the 
Carath\'eodory kernel for tamed sequences of pointed domains. 

Example~\ref{Pac-man_example} and Remark~\ref{ex-flying-saucer} can be viewed as direct applications of the next

\begin{proposition} 
\label{SomeKernels}
Let $\{(G_j,p_j)\}_{j \geq 1}$ be a sequence of domains tamed at $\hat p$. 
Then the following hold:
\begin{enumerate}
\item If $\{G_j\}_{j \geq 1}$ is increasing, i.e., $G_j \subseteq G_{j+1}$ for
every $j\geq1$, then
\[
\Ker_{\hat p} \{(G_j, p_j)\}_{j\geq1}  
= \Big(\bigcup_{j \geq 1} G_j,\hat{p} \Big).
\]
\item If $\{G_j\}_{j \geq 1}$ is decreasing, i.e., $G_{j+1} \subseteq G_{j}$ 
for every $j\geq1$, then
\[
\Ker_{\hat p} \{(G_j, p_j)\}_{j\geq1}  = \Big(\operatorname{Conn}_{\hat{p}}
\big(\bigcap_{j \geq 1} G_j\big)^\circ,\hat{p} \Big).
\]
\end{enumerate}
\end{proposition}

\begin{proof}
The proof is straight-forward and follows from the definition directly.
\end{proof}

\begin{remark} 
If the sets $G_j$ are domains of holomorphy (= pseudoconvex domains), 
then it is well known that the kernels in Proposition~\ref{SomeKernels} 
are pseudoconvex, by the Behnke-Stein theorem and the 
Cartan-Thullen theorem. In general, a tamed sequence of pointed 
domains may not be monotone. 
Nevertheless, if the members of the sequence are pseudoconvex,
$V_k:=\operatorname{Conn}_{\hat{p}}\big(\bigcap_{j \geq k} G_j\big)^\circ$ 
is pseudoconvex for every $k\geq1$. Since the sequence $\{V_k\}_{k\geq1}$ 
is increasing, the kernel is pseudoconvex, as well. 
\end{remark}

In the graph case, we have the following

\begin{theorem} \label{graph-case}
Let the sequence of domains $\{G_j\}_{j\geq1}$ be given by
\[
G_j := \{(z_1, \ldots, z_n) \in \CC^n \colon 
\re(z_1) > \varphi_j (\im (z_1), z_2, \ldots z_n)\},
\]
where $\varphi_j \colon \Pi \to \RR$ is a $\cC^1$-smooth
function with $\varphi_j (0, \ldots, 0) = 0$ defined over the hyperplane 
$\Pi = \{(z_1, \ldots, z_n) \in \CC^n \colon \re (z_1 )=0\}$. Denote by 
$\mathbf{1} := (1,0,\ldots,0)$.  Assume that the sequence 
$\{\varphi_j\}_{j\geq1}$ converges  uniformly on compact subsets of 
$\Pi$ to $\widehat{\varphi} \colon \Pi \to \RR$.
Then the sequence $\{(G_j, \mathbf{1})\}_{j\geq1}$ of pointed domains
is tamed at $\mathbf{1}$ and converges in the sense of Carath\'eodory 
to the pointed domain
$(\widehat G, \textbf{1})$ given by $\big(\{\re (z_1) > \widehat{\varphi}\},\mathbf{1}\big)$.
\end{theorem}

The proof is straightforward, since the normal limit of the sequence 
$\{(G_j, \mathbf{1})\}_{j\geq1}$ turns out to be $(\widehat G, \textbf{1})$ 
and, moreover, the normal limits stay the same for the subsequences of 
$\{(G_j, \mathbf{1})\}_{j\geq1}$.  The statement follows also from the more
general result Theorem~\ref{prop-kernel-def-fct}.

The kernel is, in a wider sense, related to the limit infimum of sets.

\begin{definition}
Let $\{G_j\}_{j\geq1}$ be a sequence of domains in $\CC^n$.  Then define by
\[
\liminf_{j\to\infty} G_j := \bigcup_{k\ge 1} \bigcap_{j \ge k} G_j
\]
the \emph{limit infimum} of $\{G_j\}_{j\geq1}$.  Notice that, in general, 
it might not be an open set.  Define also by
\[
\preker \{G_j\}_{j\geq1} 
:= \bigcup_{k\ge 1} \Big(\bigcap_{j \ge k} G_j\Big)^\circ
\]
the \emph{pre-kernel} of the sequence $\{G_j\}_{j\geq1}$.  
The pre-kernel need not be connected in general, but it is always open.
\end{definition}

%%%%%%%%%%%%
Then we present the relation between the pre-kernel and the kernel 
in the sense of Carath\'eodory.

\begin{lemma} \label{ker_vs_preker}
Let $\{(G_j, p_j)\}_{j\geq1}$ be a pointed sequence of domains tamed at~$\hat p$, and let
$\widehat G = \preker\{G_j\}_{j\geq1}$ be 
its pre-kernel.  If $\hat p \in \widehat G$, then
\[
\Ker_{\hat p} \{(G_j,p_j)\}_{j\geq1}
= \big(\operatorname{Conn}_{\hat p} (\widehat G), \hat p\big).
\]
\end{lemma}

\begin{proof} 
Let $z \in \operatorname{Conn}_{\hat p} (\widehat G)$ and recall that a 
connected open set in $\CC^n$ is always path-connected.  Then there is 
a path $\gamma$ in $\widehat G$ connecting $z$ and $\hat p$. By the 
definition of the pre-kernel, there is an index $k_0\geq1$ such that
 $\gamma$ is contained in the interior of $\bigcap_{j\ge k_0} G_j$.  
Thus,
the image of $\gamma$ lies inside the connected component of 
$(\bigcap_{j \geq k_0} G_j)^\circ$ containing $\hat p$.  
Therefore, we obtain
\[
\big(\operatorname{Conn}_{\hat p} (\widehat G), \hat p\big) 
\subseteq \Ker_{\hat p} \{(G_j,p_j)\}_{j\geq1}.
\]
For the reverse inclusion notice that
\[
\hat p \in \operatorname{Conn}_{\hat p} 
\Big( \bigcap_{j \ge k_0} G_j \Big)^\circ \subseteq 
\Big( \bigcap_{j \ge k_0} G_j \Big)^\circ, 
\]
implies
\[
\hat p \in \bigcup_{k_0 \ge 1} \operatorname{Conn}_{\hat p} 
\Big( \bigcap_{j \ge k_0} G_j \Big)^\circ
\subseteq 
\bigcup_{k_0 \ge 1} \Big( \bigcap_{j \ge k_0} G_j \Big)^\circ,
\]
which in turn leads to 
\begin{align*}
\Ker_{\hat p} \{(G_j,p_j)\}_{j\geq1} 
&\subseteq \Big(\operatorname{Conn}_{\hat p} 
\bigcup_{k_0 \ge 1} \Big( \bigcap_{j \ge k_0} G_j \Big)^\circ, \hat p\Big) \\
&= \big(\operatorname{Conn}_{\hat p} (\widehat G), \hat p\big).
\end{align*}
This yields the desired conclusion. 
\end{proof}
\bigskip

\begin{theorem}\label{prop-kernel-def-fct} 
Let $\{\psi_j\}_{j\geq1}$ be a sequence of upper semi-continuous 
functions on some domain $D$ in $\CC^n$, and let $\{p_j\}_{j\geq1}$ be
a sequence of points in $D$ converging to $\hat p$. 
Assume that $\big\{(\{\psi_j<0\}, p_j)\big\}_{j\geq1}$ is 
a sequence of pointed domains tamed at $\hat p$.
Define by $\Psi$ the function
\[
\Psi(z) :=\inf_{k\geq1} \big(\sup_{j \geq k} \psi_j\big)^* (z), \ z \in D,
\]
where $f^*$ denotes the 
upper semi-continuous regularization of~$f$ defined by 
$f^*(w):=\limsup_{\zeta \to w} f(\zeta)$.  If $\{ \Psi < 0\}$ is
connected and contains $\hat p$, and if additionally $\{ \Psi \leq 0\}^\circ = \{ \Psi < 0\}$, then 
\[
\Ker_{\hat p} \big\{(\{\psi_j<0\}, p_j)\big\}_{j\geq1} =\big(\{ \Psi < 0\}, \hat{p}\big).
\]
\end{theorem}

\begin{proof}
Notice that, by Lemma~\ref{ker_vs_preker}, it suffices to show that, 
for the pre-kernel, it holds 
\[
\preker \{\psi_j<0\}_{j\geq1}
=\{ \Psi < 0\}.
\]
Let $z_0$ be contained in the pre-kernel 
\[
\preker \{\psi_j<0\}_{j\geq1} 
= \bigcup_{k \geq 1}\big( \bigcap_{j\geq k} \{\psi_j <0\}\big)^\circ.
\]
Recall that such an element $z_0$ exists due to the tameness of the 
sequence. Then there is an index $k_0$ such that the open ball 
$B:=B_R(z_0)$ is contained in $\{\psi_j <0\}$ for every $j \geq k_0$. 
Since $\psi_j<0$ on $B$ for every $j\geq k_0$, we have that 
$\sup_{j\geq k_0} \psi_j \leq 0$ on $B$. 
Hence,  $\big(\sup_{j\geq k_0}\psi_j\big)^* \leq 0$ on $B$. Since 
$k \mapsto \big(\sup_{j\geq k}\psi_j\big)^*$ is decreasing for 
$k\geq1$, it holds for any $z \in B$ that
\[
\Psi(z) = \inf_{k\geq 1}\big(\sup_{j\geq k}\psi_j\big)^*(z) 
= \inf_{k\geq k_0}\big(\sup_{j\geq k}\psi_j\big)^*(z) \leq 0.
\] 
Therefore, $B$ lies in $\{\Psi\leq 0\}$. But then $z_0$ lies inside the 
interior of $\{\Psi\leq 0\}$. By assumption, 
$\{ \Psi \leq 0\}^\circ = \{ \Psi < 0\}$, so $z_0 \in \{ \Psi < 0\}$. 
Since $z_0$ was arbitrarily chosen from the pre-kernel, 
we conclude that the pre-kernel of 
$\{\psi_j<0\}_{j\geq1}$ is contained in $\{\Psi< 0\}$.
\smallskip

Now let us assume that $w_0 \in \{\Psi<0\}$. 
Since $\Psi$ is upper semi-continuous, 
the set $\{\Psi <0\}$ is open and we can find a ball $B_r(w_0)$ fully 
contained in $\{\Psi <0\}$. Define by $K$ the closure of a slightly 
smaller ball $B_s(w_0)$, where $s<r$. Since $K$ is compact and 
$\Psi$ is upper semi-continuous, the function $\Psi$ attains 
a maximum on $K$, so there is a real number $M$ such that
\[
\Psi < M < 0 \ \text{on}\ K.
\]
Since $k \mapsto \big(\sup_{j\geq k}\psi_j\big)^*$ is 
a decreasing sequence of 
upper semi-continuous functions converging to $\Psi$, and since 
$K$ is compact, there is an index $k_1\geq1$ such that 
$\big(\sup_{j\geq k}\psi_j\big)^* <M$ for every $k \geq k_1$. But this means 
that for $z \in K$ and every $j \geq k \geq k_1$ we have
\[
\psi_j(z) \leq \sup_{j\geq k}\psi_j(z) 
\leq \big(\sup_{j\geq k}\psi_j\big)^*(z) < M <0.
\]
Hence, $K$ lies in $\{\psi_j<0\}$ for each $j \geq k_1$, so $w_0$ has to be 
an interior point of $\bigcap_{j\geq k_1} \{\psi_j<0\}$. Therefore, $w_0$ lies 
in the pre-kernel of $\{\psi_j<0\}_{j \geq 1}$ according to its definition. 
Since $w_0$ was chosen arbitrarily from $\{\Psi<0\}$, the whole set 
$\{\Psi<0\}$ is contained in the pre-kernel of $\{\psi_j<0\}_{j \geq 1}$. 
This completes the proof. 
\end{proof}
\bigskip

\begin{remark} 
The condition $\{ \Psi \leq 0\}^\circ = \{ \Psi < 0\}$ 
in Theorem~\ref{prop-kernel-def-fct} is equivalent to
$\partial \{ \Psi \leq 0\} = \{\Psi=0\}$. It naturally occurs, for instance, 
if $\Psi$ is $\mathcal{C}^1$-smooth and $\nabla \Psi(z) \neq 0$ for 
any boundary point of $\{\Psi<0\}$ in $D$, or if $\Psi$ is 
strictly plurisubharmonic. Also notice that, if all the $\psi_j$'s are 
plurisubharmonic, then 
$\Psi=\inf_{k\geq1} \big(\sup_{j \geq k} \psi_j\big)^*$ 
is plurisubharmonic as well.
\end{remark}
\medskip

Notice that, if $\{\psi_j\}_{j\geq1}$ converges locally uniformly to $\Psi$, 
then for \emph{any} subsequence $\{\psi_{j_\ell}\}_{\ell\geq1}$, the function 
$\inf_{k\geq1} \big(\sup_{\ell \geq k} \psi_{j_\ell}\big)^*$ equals
$\Psi$. Therefore, we can extend the previous result:

\begin{corollary} 
Under the assumptions of Theorem~\ref{prop-kernel-def-fct}, 
the sequence of pointed domains $\{(\{\psi_j>0\},p_j)\}_{j\geq1}$ 
tamed at $\hat p$ converges in the sense of Carath\'eodory to 
$(\{\Psi<0\},\hat{p})$, if $\{\psi_j\}_{j\geq1}$ converges 
locally uniformly to $\Psi$.
\end{corollary}

\section{Final remarks}

Notice that our investigations of the Carath\'eodory kernel convergence 
differ from that of the Hausdorff convergence of the sequences of
compact subsets.  Our analyses do not depend upon any particular
distance concepts.  In this regard, it may be interesting to investigate
where there is an appropriate distance inducing the Carath\'eodory kernel 
convergence.
\medskip

At the final stage of this writing, the authors became aware of 
\cite{BRY}, which also 
studied the concept of Carath\'eodory kernel convergence.  
However, we realized that the main goal and the
analyses differ from ours. In their article \cite{BRY}, the authors gave characterizations on the 
existence of the kernel convergence using the harmonic measures, 
whereas, in this article, we restrict to a purely topological characterization 
of the kernel convergence and relate it to families of biholomorphic 
mappings for one and higher dimensions.

\backmatter

\section*{Declarations}

This work was supported by the National Research Foundation (NRF) grant
(No. 2023R1A2C1007227) and by Samsung Science and Technology 
Foundation under Project Number SSTF-BA2201-01, of The Republic of
Korea. %The authors have no relevant financial or non-financial interests to

\bibliography{sn-bibliography}% common bib file
%% if required, the content of .bbl file can be included here once bbl is generated
%%\input sn-article.bbl

\end{document}